\RequirePackage[l2tabu, orthodox]{nag}

\documentclass[12pt,reqno]{amsart}
\usepackage{fullpage,url,amssymb,enumerate,colonequals,pdfsync}
\usepackage{mathrsfs} 

\usepackage{color}
\usepackage[dvipsnames]{xcolor}

\newcommand{\defi}[1]{\textsf{#1}} 

\newcommand{\C}{\mathbb{C}}
\newcommand{\F}{\mathbb{F}}

\newcommand{\Q}{\mathbb{Q}}
\newcommand{\R}{\mathbb{R}}
\newcommand{\Z}{\mathbb{Z}}

\newcommand{\Kbar}{{\overline{K}}}

\newcommand{\mm}{\mathfrak{m}}

\newcommand{\calD}{\mathcal{D}}

\newcommand{\calO}{\mathcal{O}}

\DeclareMathOperator{\Aut}{Aut}

\DeclareMathOperator{\Char}{char}

\DeclareMathOperator{\Gal}{Gal}

\DeclareMathOperator{\sgn}{sgn}

\newcommand{\injects}{\hookrightarrow}
\newcommand{\intersect}{\cap} 
\newcommand{\Intersection}{\bigcap} 
\newcommand{\isom}{\simeq}

\newcommand{\surjects}{\twoheadrightarrow}
\newcommand{\To}{\longrightarrow}
\newcommand{\union}{\cup} 
\newcommand{\Union}{\bigcup} 

\newcommand{\isomto}{\overset{\sim}{\rightarrow}}

\newtheorem{theorem}{Theorem}[section]
\newtheorem{lemma}[theorem]{Lemma}
\newtheorem{corollary}[theorem]{Corollary}

\theoremstyle{definition}

\newtheorem{example}[theorem]{Example}

\theoremstyle{remark}
\newtheorem{remark}[theorem]{Remark}

\makeatletter
\g@addto@macro\bfseries{\boldmath} 
\makeatother

\usepackage{microtype}

\usepackage[
	backref,
	pdfauthor={Jacqueline Anderson, Spencer Hamblen, Bjorn Poonen, Laura Walton}, 
	pdftitle={Local arboreal representations},
]{hyperref}
\usepackage[alphabetic,backrefs,lite,nobysame]{amsrefs} 

\begin{document}

\title{Local arboreal representations}

\author{Jacqueline Anderson}
\address{Mathematics Department, Bridgewater State University, Bridgewater, MA 02325, USA}
\email{jacqueline.anderson@bridgew.edu}

\author{Spencer Hamblen}
\address{McDaniel College, 2 College Hill, Westminster, MD 21157}
\email{shamblen@mcdaniel.edu}

\author{Bjorn Poonen}
\thanks{B.P.\ was supported in part by National Science Foundation grants DMS-1069236 and DMS-1601946 and grants from the Simons Foundation (\#340694 and \#402472 to Bjorn Poonen).  This article was originally published in \href{https://doi.org/10.1093/imrn/rnx054}{\emph{IMRN} \textbf{2018}, no.~19, 5974--5994}.}
\address{Department of Mathematics, Massachusetts Institute of Technology, Cambridge, MA 02139-4307, USA}
\email{poonen@math.mit.edu}
\urladdr{\url{http://math.mit.edu/~poonen/}}

\author{Laura Walton}
\address{Mathematics Department, Brown University, Providence, RI 02912, USA}
\email{laura@math.brown.edu}
\date{March 6, 2017}

\begin{abstract}
Let $K$ be a field complete with respect to a discrete valuation $v$
of residue characteristic~$p$.
Let $f(z) \in K[z]$ be a separable polynomial of the form $z^\ell-c$.
Given $a \in K$, we examine the Galois groups and ramification groups
of the extensions of $K$ generated by the solutions to $f^n(z)=a$.
The behavior depends upon $v(c)$, and we find that it shifts dramatically
as $v(c)$ crosses a certain value: $0$ in the case $p \nmid \ell$, 
and $-p/(p-1)$ in the case $p=\ell$.
\end{abstract}

\maketitle

\section{Introduction}\label{S:introduction}

\subsection{Arboreal Galois representations}
\label{S:1.1}

Let $K$ be a field.
Choose an algebraic closure $\Kbar$.
Let $f(z)$ be a polynomial of degree $\ell$ over $K$.
For $n \ge 0$, let $f^n$ denote the $n$th iterate $f \circ f \circ \cdots \circ f$.
Fix $a \in K$.
For $n \ge 0$,
let $f^{-n}(a)$ be the multiset of solutions to $f^n(z)=a$ in $\Kbar$,
so $\#f^{-n}(a)=\ell^n$;
also let $K_n = K(f^{-n}(a)) \subseteq \Kbar$.
Let $K_\infty = \Union_{n \ge 1} K_n$.
For $0 \le n \le \infty$, 
let $G(n) = \Aut(K_n/K)$.

Let $n \in \{0,1,2,\ldots,\infty\}$.
Let $T_n$ be the complete $\ell$-ary rooted tree of height $n$ (so there are $\ell^n$ leaves at the top);
here $T_\infty$ is the increasing union of $T_1 \subset T_2 \subset \cdots$.
The disjoint union of the $f^{-m}(a)$ for $m \le n$,
with an edge from $\alpha$ to $f(\alpha)$ for each vertex $\alpha$ other than the root,
is isomorphic to $T_n$.
\emph{For the rest of the paper, we suppose that for each $n \in \Z_{\ge 0}$, the solutions to $f^n(z)=a$ are distinct.}
Then these solutions lie in the separable closure $K_s$ of $K$ in $\Kbar$, 
and $\Gal(K_s/K)$ acts on this copy of $T_n$.
This defines a continuous homomorphism $\rho_n \colon \Gal(K_s/K) \to \Aut T_n$.
The image of $\rho_n$ is isomorphic to $G(n)$.
A continuous homomorphism $\Gal(K_s/K) \to \Aut T_\infty$
is called an \defi{arboreal Galois representation} \cite{Boston-Jones2007}*{Definition~1.1}.

There is a large literature studying the image of $\rho_\infty$ 
for various polynomials over global fields \cites{Odoni1985a,Odoni1985b,Stoll1992,Odoni1997,Boston-Jones2007,Jones2008,Boston-Jones2009,Jones2013,Hindes2016},
and occasionally also for rational functions \cite{Jones-Manes2014}.

\begin{example}
\label{E:Odoni}
Let $K=\Q$ and $f(z)=z^2-z+1$ and $a=0$.
Then $\rho_\infty$ is surjective \cite{Odoni1985a}*{Theorem~1}.
\end{example}

\begin{example}
\label{E:Stoll}
Let $K=\Q$.
Let $b \in \Z$ be such that either $b>0$ and $b \equiv 1,2 \pmod 4$,
or $b<0$ and $b \equiv 0 \pmod 4$ and $-b$ is not a square.
Let $f(z)=z^2+b$ and $a=0$.
Then $\rho_\infty$ is surjective \cite{Stoll1992}.
\end{example}

\subsection{Local fields}
\label{S:local fields}

{}From now on, $K$ is a field that is complete with respect to a discrete valuation $v$.
Let $k$ be the residue field.
Let $p$ be the characteristic of $k$.
Extend $v$ to $K_s$.

Consider $f(z) \colonequals z^\ell-c \in K[z]$ for some $\ell \ge 2$
and $c \in K^\times$.
Outside Section~\ref{S:image of infinite index}, we assume additionally that we are in one of the following cases:
\begin{itemize}
\item (``Tame case'') $\ell$ is not divisible by $p$;
\item (``Wild case'') $\ell=p$ and $K$ is a finite extension of $\Q_p$; in this case we normalize $v$ so that $v(p)=1$.
\end{itemize}
In particular, $f$ is separable.

In contrast with the situation over global fields in Examples \ref{E:Odoni} and~\ref{E:Stoll},
our Theorem~\ref{T:infinite index} will imply that
over a local field $K$ with finite residue field, 
the arboreal representation associated to a separable polynomial $f(z)=z^\ell-c$
as above is \emph{never} surjective, and never even of finite index.
Ingram proved a related result when $K$ is a finite extension of $\Q_p$. In this setting, he showed that if
$f \in K[x]$ is a monic polynomial with good reduction and degree not divisible by $p$,
and $a \in K$ is such that $f^n(a) \to \infty$ as $n \to \infty$, 
then the image of $\Gal(K_s/K)$ is of finite index in a particular
infinite index subgroup of $\Aut T_\infty$ \cite{Ingram2013}*{Theorem~1}.

In this introduction, we describe our main results in the wild case; 
the results in the tame case are similar but easier.
It turns out that in the wild case there is a dramatic shift of behavior as $v(c)$ crosses $-p/(p-1)$:

\begin{theorem}
Suppose that $K$ is a finite extension of $\Q_p$, and $\ell=p$.
\begin{enumerate}[\upshape (a)]
\item If $v(c) < -p/(p-1)$, then $K_\infty/K$ is a finite extension.
\item If $v(c) = -p/(p-1)$, then $K_\infty/K$ is an infinite extension, 
and $K_\infty/K$ is finitely ramified if and only if 
$a$ lies within the closed unit disk centered at a fixed point of $f$.
\item If $v(c) > -p/(p-1)$, then $K_\infty/K$ is infinitely wildly ramified.
\end{enumerate}
\end{theorem}

In fact, our results are more precise.
For example:
\begin{itemize}
\item
If $v(c) < -p/(p-1)$ and $v(a) > v(c)/p$ and $\mu_p \subseteq K$, 
then there exists $n$ depending on $v(c)$ and there exists $\alpha \in f^{-n}(a)$
such that $K_\infty = K(\alpha)$ (generated by one element!)\ 
and $G(\infty)$ is an elementary abelian $p$-group of order at most $p^n$
(Theorems \ref{T:elementary abelian p-group} and \ref{T:cutoff points}).
\item
If $v(c) = -p/(p-1)$, then some upper numbering ramification subgroup 
of $G(\infty)$ is trivial (Theorem~\ref{T:trivial ramification group for -p/(p-1)};
see also Example~\ref{E:shallow ramification}).
This contrasts with Sen's filtration theorem: see Remark~\ref{R:Sen}.
\item
If $v(c) = -p/(p-1)$ and $v(a)>v(c)$ and $\mu_p \subset K$, 
then the inertia subgroup $I(\infty)$ of $G(\infty)$
is either $\{1\}$ or $(\Z/p\Z)^\infty$ (Theorem~\ref{T:-p/(p-1)}).
\item
If $v(c) < 0$, Theorem~\ref{T:group size} provides a nontrivial upper bound on 
the asymptotic rate of growth of $[K_n:K]$.
\end{itemize}
The lack of deep ramification, at least when $v(c) \le -p/(p-1)$, 
contrasts with the expectation in 
an early study of ramification in arboreal representations \cite{Aitken-Hajir-Maire2005}*{p.~858}
that preimage trees of a generic polynomial of degree divisible by $p$ should be deeply ramified; 
see also \cite{Cullinan-Hajir2012} 
for other results on ramification in arboreal representations, also for rational functions.

\begin{remark}
Given $f$ over a global field $K$, the images of the associated local arboreal representations
give lower bounds on the global arboreal representation.
One might hope that these could be used to prove surjectivity of the global arboreal representation,
but so far the arguments in the literature that have been used to prove global surjectivity
(such as in \cite{Stoll1992}) have used a mix of local and global arguments.
\end{remark}

\subsection{Outline of the paper}

Section~\ref{S:image of infinite index} shows that the image of an arboreal representation
over a local field has infinite index, whether or not it arises from iterates of a polynomial.
Section~\ref{S:general lemmas} proves some general lemmas used throughout the rest of the paper.
The Galois groups $G(n)$ and $G(\infty)$ depend on whether $v(c)$ is negative,
and in the wild case also on whether $v(c) < -p/(p-1)$.
Sections \ref{S:sufficiently negative valuation} to~\ref{S:nonnegative valuation}
describe these groups; the section titles refer to the valuation of $c$.
Finally, in Section~\ref{S:real case}, 
we determine $K_\infty$ completely 
in the analogous situation with $K = \R$.

\section{Images of local arboreal representations}
\label{S:image of infinite index}

\begin{theorem}
\label{T:infinite index}
Let $K$ be a field that is a complete with respect to a discrete valuation $v$ with finite residue field $k$.
Assume that $\Char K \ne 2$.
Let $d \ge 2$, and let $T_\infty$ be the infinite $d$-ary rooted tree defined in Section~\ref{S:1.1}.
Then the image of any continuous homomorphism $\rho_\infty \colon \Gal(K_s/K) \to \Aut T_\infty$ 
is of infinite index.
\end{theorem}

\begin{proof}
Each $\tau \in \Aut T_\infty$ acts as a permutation of the set of the leaves of $T_n$;
let $\sgn_n(\tau)$ be the sign of this permutation.
We define a map $\sgn\colon \Aut T_\infty\to \prod_{n\ge1} \{\pm 1\}$ by assigning $\tau\mapsto \prod_{n\ge1} \sgn_n(\tau)$.

The hypotheses on $K$ imply that $K$ has only finitely many quadratic extensions.
These are in bijection with the surjective continuous homomorphisms $\Gal(K_s/K) \to \{\pm 1\}$,
so there are only finitely many such homomorphisms.
Thus the composition 
\[
\Gal(K_s/K) \stackrel{\rho_\infty}\to \Aut T_\infty \stackrel{\sgn}\to \prod_{n \ge 1} \{\pm 1\}
\]
factors through a finite product of copies of $\{\pm 1\}$, and hence has finite image.
On the other hand, the map $\Aut T_\infty \stackrel{\sgn}\to \prod_{n \ge 1} \{\pm 1\}$ is surjective.
\end{proof}

\begin{remark}
Without the assumption that $k$ is finite, Theorem~\ref{T:infinite index} can fail.
For example, if $K=\Q((t))$ and $d=2$, then any $f(x)$ as in Example~\ref{E:Stoll} defines a surjective $\rho_\infty$.
\end{remark}

\begin{remark}
If $k$ is finite but $\Char K=2$, then again Theorem~\ref{T:infinite index} can fail,
as we now explain.
In this case, $K=\F_{2^e}((t^{-1}))$ for some $e$,
and the maximal pro-$2$ quotient of $\Gal(K_s/K)$
is a free pro-$2$ group of infinite rank \cite{Katz1986}*{1.4.4}.
This implies that $\Gal(K_s/K)$ admits a continuous surjective homomorphism
onto any inverse limit of a sequence of finite $2$-groups.
If $T_\infty$ is a binary tree ($d=2$), 
then $\Aut T_\infty$ is such an inverse limit.
\end{remark}

\section{General lemmas}
\label{S:general lemmas}

For $n \ge 1$, 
let $\nu_n = -\frac{\ell^{n+1}}{(\ell^n - 1)(\ell-1)} v(\ell)$.
Let $\nu_\infty = -\frac{\ell}{\ell-1} v(\ell)$.
It will turn out that there is a shift of behavior when $v(c)$ crosses
these values.
In the tame case, all these values collapse into one: $\nu_n = 0$ for all $n \le \infty$.
In the wild case, $\nu_n = -\frac{p^{n+1}}{(p^n - 1)(p-1)}$ and their limit is $\nu_\infty = -\frac{p}{p-1}$.

\begin{lemma}
\label{L:differences between preimages}
Let $d,y \in \Kbar$.
Consider the $\ell$ solutions $x$ to $f(x)-f(y)=d$, counted with multiplicity.
\begin{enumerate}[\upshape (a)]
\item \label{I:small v(d)}
If $v(d) \le \ell v(y) - \nu_\infty$, then $v(x-y) = v(d)/\ell$ for each $x$.
\item \label{I:large v(d)}
If $v(d) > \ell v(y) - \nu_\infty$, then the solution $x$ that is closest to $y$ satisfies 
$v(x-y)=v(d) - (\ell-1) v(y) - v(\ell)$ and the other $(\ell-1)$ solutions $x$ satisfy $v(x-y) = v(y) + v(\ell)/(\ell-1)$.
The first solution lies in $K(d,y)$.
\item \label{I:v(d) just right}
If $\ell=p$ and $v(d) = \ell v(y) - \nu_\infty$, 
then the solutions generate an unramified extension of $K(d,y)$.
\end{enumerate}
\end{lemma}

\begin{proof}\hfill
\begin{enumerate}[\upshape (a)]
\item[(a,b)]
Let $z=x-y$.
Let $K'=K(d,y)$.
We need the valuations of the zeros of the polynomial 
\[
   f(z+y) - f(y) - d 
   = z^\ell + \binom{\ell}{1} y z^{\ell-1} + \binom{\ell}{2} y^2 z^{\ell-2} + \cdots + \binom{\ell}{\ell-1} y^{\ell-1} z - d \in K'[z].
\]
Its Newton polygon is the lower convex hull of the points $(0,v(d))$, $(1,(\ell-1)v(y) + v(\ell))$, and $(\ell,0)$.
The slopes of the Newton polygon depend on whether the middle point
lies above or below the line segment through $(0,v(d))$ and $(\ell,0)$.
These slopes determine the valuations of the zeros.
A Newton polygon segment of width~$1$ corresponds to a solution in the ground
field $K(d,y)$.
\item[(c)]
The Newton polygon of $f(z+y)-f(y)-d$ 
is a line segment containing the three points above, while all other
intermediate monomials correspond to points strictly above this line
since the prime $\ell$ divides each binomial coefficient.
Thus, if we scale the variable to make the first two points horizontal,
and then divide by the leading coefficient, we obtain a polynomial $g(z)$
reducing to 
$\bar{g}(z) \colonequals z^\ell + u_1 z + u_2$ for some units $u_1,u_2$.
We have $\bar{g}'(z) = u_1$, so $\bar{g}$ is separable,
so the roots of $g$ generate an unramified extension.\qedhere
\end{enumerate}  
\end{proof}

\begin{lemma}
\label{L:valuations of preimages}
Suppose that $v(c)<0$.
If $n$ is sufficiently large, then every $\alpha \in f^{-n}(a)$
satisfies $v(\alpha)=v(c)/\ell$.
If $v(a)>v(c)$, then this conclusion holds for all $n \ge 1$.
\end{lemma}

\begin{proof}
Let $\alpha_0=a$ and let $\alpha_{n+1} \in f^{-1}(\alpha_n)$ for $n \ge 1$.
The equation $\alpha_{n+1}^\ell = \alpha_n + c$ implies that 
\[
   v(\alpha_{n+1}) = \begin{cases}
                        v(\alpha_n)/\ell, & \textup{ if $v(\alpha_n)<v(c)$;} \\
                        v(c)/\ell \textup{ or larger}, & \textup{ if $v(\alpha_n)=v(c)$;} \\
                        v(c)/\ell, & \textup{ if $v(\alpha_n)>v(c)$.}
\end{cases}
\]
Thus the first case holds at most finitely many times, and then the second case holds at most once,
and then the third case holds from then on.
\end{proof}

\begin{lemma}
\label{L:G(n) is a p-group}
If $\mu_\ell \subset K$, then $\#G(n)$ divides a power of $\ell$.
\end{lemma}

\begin{proof}
Each extension $K_{n+1}/K_n$ is a Kummer extension of exponent dividing $\ell$.
\end{proof}

\section{Sufficiently negative valuation}
\label{S:sufficiently negative valuation}

In this section, we consider the case $v(c) < \nu_\infty$.
Recall that $\nu_\infty = -\frac{\ell}{\ell-1} v(\ell)$.

\begin{lemma}
\label{L:differences between zeros at the same level}
Suppose that $v(c)<\nu_\infty$ and $v(a)>v(c)$.
If $n \ge 0$ and $\alpha,\beta \in f^{-n}(a)$, then $v(\alpha-\beta) \ge v(c)/\ell + v(\ell)/(\ell-1)$.
\end{lemma}

\begin{proof}
We may assume that $n \ge 1$ and $\alpha \ne \beta$.
We use induction on $n$.
If $n=1$, then $\beta^\ell=c+a$, so $v(\beta) = v(c+a)/\ell = v(c)/\ell$.
Also $\alpha^\ell=c+a$, 
so $\alpha = \zeta \beta$ for some $\ell$th root of unity $\zeta$.
Then $v(\alpha-\beta)=v((\zeta-1)\beta) = v(\ell)/(\ell-1) + v(c)/\ell$.

Suppose that $n>1$ and the result holds for $n-1$.
Let $d=f(\alpha)-f(\beta)$ and $y=\beta$.
If $n>1$, then by the inductive hypothesis, the hypothesis on $c$, and Lemma~\ref{L:valuations of preimages},
\begin{equation}
\label{E:chain of inequalities}
v(d) \ge v(c)/\ell + v(\ell)/(\ell-1) > v(c) + \ell v(\ell)/(\ell-1) = \ell v(y) - \nu_\infty,
\end{equation}
so Lemma~\ref{L:differences between preimages}\eqref{I:large v(d)}
shows that $v(\alpha-\beta) \ge v(y) + v(\ell)/(\ell-1) = v(c)/\ell + v(\ell)/(\ell-1)$.
\end{proof}

For $n \le \infty$, let $I(n)$ be the inertia subgroup of $G(n)$.

\begin{theorem}
\label{T:elementary abelian p-group}
If $v(c)<\nu_\infty$ and $v(a)>v(c)$ and $\mu_\ell \subseteq K$, then
\begin{enumerate}[\upshape (a)]
\item \label{I:G_n}
The group $G(n)$ is isomorphic to a subgroup of $(\Z/\ell\Z)^n$.
\item \label{I:G_n/I_n}
If the residue field of $K$ is finite, then the group $G(n)/I(n)$ is cyclic of order dividing $\ell$.
\item \label{I:generated by one element}
For any $\alpha_n \in f^{-n}(a)$, we have $K_n=K(\alpha_n)$.
\end{enumerate}
\end{theorem}

\begin{proof}\hfill
\begin{enumerate}[\upshape (a)]
\item
Let $\delta=v(c)/\ell+v(\ell)/(\ell-1)$.
Let $m  \in \{1,\ldots,n\}$.
For $x,y \in f^{-m}(a)$, write $x \sim y$ if $v(x-y)>\delta$; this defines an equivalence relation.
Let $\calD_m = f^{-m}(a)/\sim$.
Suppose that $\alpha_{m-1},\beta_{m-1} \in f^{-(m-1)}(a)$ 
and $\alpha_m \in f^{-1}(\alpha_{m-1})$.
By Lemma~\ref{L:differences between zeros at the same level},
$v(\beta_{m-1}-\alpha_{m-1}) \ge \delta$.
Lemma~\ref{L:differences between preimages}\eqref{I:large v(d)}
with $(d,y) \colonequals (\beta_{m-1}-\alpha_{m-1},\alpha_m)$
applies (by \eqref{E:chain of inequalities}),
so for all but one $\beta_m \in f^{-1}(\beta_{m-1})$,
we have $v(\beta_m-\alpha_m) = v(y) + v(\ell)/(\ell-1) = \delta$,
and for the other $\beta_m$, we have $v(\beta_m-\alpha_m) > \delta$.
In other words,
exactly one preimage of $\beta_{m-1}$ is equivalent to $\alpha_m$.
Thus the map 
\begin{align*}
   f^{-m}(a) &\To f^{-(m-1)}(a) \times \calD_m \\
   x &\longmapsto (f(x),\textup{equivalence class of $x$})
\end{align*}
is a bijection.
The multiplication action of $\mu_\ell$ on $f^{-m}(a)$ 
is compatible with the trivial action on $f^{-(m-1)}(a)$;
on the other hand, it induces an action on $\calD_m$.
The action on $f^{-m}(a)$ is free 
(since the elements of $f^{-m}(a)$ are nonzero),
so the action on $\calD_m$ is free.
But $\# \calD_m = \ell^m / \ell^{m-1} = \ell = \# \mu_\ell$,
so $\calD_m$ is a $\mu_\ell$-torsor, 
and its automorphism group as a torsor is $\mu_\ell$ 
(for any group $H$, the automorphism group of a left $H$-torsor is isomorphic to $H$ acting on the right).
Each element of $G(n)$ acts trivially on $\mu_\ell$,
and hence acts as an automorphism of the $\mu_\ell$-torsor $\calD_m$.
Combining the bijections for $m=1,\ldots,n$ yields a Galois-equivariant bijection
$f^{-n}(a) \isomto \prod_{i=1}^n \calD_i$,
so $G(n) \le \prod_{i=1}^n \Aut_{\textup{$\mu_\ell$-torsor}}(\calD_i) = \mu_\ell^n \isom (\Z/\ell\Z)^n$.
\item
The group $G(n)/I(n)$ is isomorphic to the Galois group of the
residue field extension, which is cyclic.
Its order divides the exponent of $G(n)$, which by \eqref{I:G_n} is $\ell$.
\item
If an element of $\prod_{i=1}^m \Aut_{\textup{$\mu_\ell$-torsor}}(\calD_i)$ fixes one element of $\prod_{i=1}^n \calD_i$, it fixes all elements.
Thus the subgroup of $G(n)$ fixing $\alpha_n$ is trivial.
By Galois theory, $K(\alpha_n)=K_n$.\qedhere
\end{enumerate}
\end{proof}

Recall that $\nu_n = -\frac{\ell^{n+1}}{(\ell^n - 1)(\ell-1)} v(\ell)$,
which is $0$ in the tame case.

\begin{theorem}
\label{T:cutoff points}
Suppose that $v(a) \ge v(c)/\ell$.
In the tame case, if $v(c) < 0$, then $K_\infty = K_1$.
In the wild case, if $v(c) < \nu_n$, then $K_\infty = K_n$,
and if $v(c) = \nu_n$, then $K_\infty=K_{n+1}$ and $K_{n+1}/K_n$ is unramified.
\end{theorem}

\begin{proof}
First suppose that $v(c) < \nu_n$.
Let $\alpha_0=a$, and for $m \ge 1$, let $\alpha_m$ be an element of $f^{-1}(\alpha_{m-1})$
minimizing the distance to $\alpha_{m-1}$.
Let $q_m = v(\alpha_m-\alpha_{m-1})$.
By Lemma~\ref{L:valuations of preimages}, $v(\alpha_m)=v(c)/\ell$ for all $m \ge 1$.
Thus $q_1 \ge v(c)/\ell$.
For $m \ge 1$, Lemma~\ref{L:differences between preimages} applied to $d=\alpha_m-\alpha_{m-1}$
and $y=\alpha_m$ implies
\begin{equation}
\label{E:q recursion}
   q_{m+1} = \begin{cases}
             q_m/\ell, & \textup{ if $q_m \le v(c) - \nu_\infty$;} \\
             q_m - (\ell-1)v(c)/\ell - v(\ell), & \textup{ otherwise.} 
             \end{cases}
\end{equation}
If the first case in \eqref{E:q recursion} holds for $m=1,2,\ldots,n-1$,
then $q_{n-1} = \ell^{1-n} q_1 \ge \ell^{-n} v(c) >  v(c) - \nu_\infty$
by definition of $\nu_n$, so the second case holds for $m=n$.
Moreover, if the second case holds for a given $m$, 
then we remain in the second case from then on, 
since $-(\ell-1)v(c)/\ell - v(\ell)$ is positive under the hypothesis $v(c) < \nu_n \le \nu_\infty$.
Thus the second case holds for all $m \ge n$, and we have $n=1$ in the tame case.
The final sentence of Lemma~\ref{L:differences between preimages}\eqref{I:large v(d)}
implies that for all $m \ge n$,
we have $\alpha_{m+1} \in K(d,y) \subseteq K_m$.
By Theorem~\ref{T:elementary abelian p-group}\eqref{I:generated by one element}, 
this implies that $K_{m+1}=K_m$ for all $m \ge n$.
Thus $K_\infty=K_n$.

Now suppose instead that we are in the wild case and $v(c)=\nu_n$.
Then $v(c)<\nu_{n+1}$, so the previous paragraph shows that $K_\infty=K_{n+1}$.
The arguments above show that if the first case holds 
for $m=1,2,\ldots,n-1$, then $q_{n-1} \ge v(c) - \nu_\infty$.
Thus we obtain $K_{n+1}=K_n$ as before unless if $q_{n-1}=v(c)-\nu_\infty$,
in which case 
Lemma~\ref{L:differences between preimages}\eqref{I:v(d) just right} 
shows that $\alpha_{n+1}$ is unramified over $K_n$
for each $\alpha_{n+1} \in f^{-(n+1)}(a)$.
\end{proof}

\begin{corollary}
\label{C:finite extension}
If $v(c)<\nu_\infty$, then $K_\infty$ is a finite extension of $K$.
\end{corollary}

\begin{proof}
Choose $n$ such that $v(c) < \nu_n$.
By Lemma~\ref{L:valuations of preimages}, there exists an $m \geq 1$ such that every $\alpha \in f^{-m}(a)$ satisfies $v(\alpha) = v(c)/\ell$.  Apply Theorem~\ref{T:cutoff points} over $K_m$ with each $\alpha$ in place of $a$,
and take the compositum of the resulting finite extensions. 
\end{proof}

\begin{theorem}
\label{T:inertia in finite case}
Suppose that $\ell=p$ and $\mu_p \subseteq K$.
\begin{enumerate}[\upshape (a)]
\item
Suppose $\nu_{n-1} < v(c) < \nu_n$ and $v(a) > v(c)/p$.
If $v(c) \notin p v(K^\times)$,
then $G(\infty) = G(n) = I(\infty) = I(n) \isom \left(\Z/p\Z\right)^n$.
More generally, if $p^r$ is the largest power of $p$ such that
$v(c) \in p^r v(K^\times)$,
then $p^{n-r} \le \#I(n) \le \#G(n) = \#G(\infty) \le p^n$.
\item
If $v(c) = \nu_n$ and $v(a) \ge v(c)/p$, 
then $G(\infty) = G(n+1) \le (\Z/p\Z)^{n+1}$,
$I(\infty)=I(n) \le (\Z/p\Z)^n$, and $G(\infty)/I(\infty) \le \Z/p\Z$.
\end{enumerate}
\end{theorem}

\begin{proof}\hfill
\begin{enumerate}[\upshape (a)]
\item
In the proof of Theorem~\ref{T:cutoff points},
we have $v(\alpha_1)=v(c)/p$, so $q_1 = v(\alpha_1-a)=v(c)/p$.
Then by \eqref{E:q recursion}, $q_m = v(c)/p^m$ for $m=1,\ldots,n$,
since the hypothesis $\nu_{n-1} < v(c)$ implies that 
$v(c)/p^{m-1} \le v(c) - \nu_\infty$
for $m \le n$.
In particular, $\alpha_n-\alpha_{n-1}$ is an element of $K_n$
whose valuation is $q_n = v(c)/p^n$, 
so the ramification index $(v(K_n^\times):v(K^\times))$ 
is at least $p^{n-r}$.
Thus $\#I(n) \ge p^{n-r}$.
On the other hand, $I(n) \le G(n) \le (\Z/p\Z)^n$ by Theorem~\ref{T:elementary abelian p-group}\eqref{I:G_n}.
In particular, if $r=0$, then equality holds.
In any case, $K_\infty=K_n$ by Theorem~\ref{T:cutoff points}.
\item
By Theorem~\ref{T:cutoff points}, $K_\infty=K_{n+1}$, and $K_{n+1}/K_n$ is unramified.
Then $G(\infty)=G(n+1) \le (\Z/p\Z)^{n+1}$ by
Theorem~\ref{T:elementary abelian p-group}\eqref{I:G_n},
and $I(\infty)=I(n+1)=I(n) \le G(n) \le (\Z/p\Z)^n$.
Finally, $G(\infty)/I(\infty) = G(n+1)/I(n+1) \le \Z/p\Z$ by 
Theorem~\ref{T:elementary abelian p-group}\eqref{I:G_n/I_n}.\qedhere
\end{enumerate}
\end{proof}


\section{Special negative valuation: \texorpdfstring{$v(c)=-p/(p-1)$}{v(c)=-p/(p-1)}}
\label{S:special negative valuation}

In this section and the next, we consider the wild case.

\subsection{Galois groups and inertia groups}

\begin{theorem}
\label{T:-p/(p-1)}
Suppose that $\ell=p$ and $v(c) = -p/(p-1)$ and $0 \le n < \infty$.
Let $b \in \Kbar$ be a fixed point of $f(z)$.
\begin{enumerate}[\upshape (a)]
\item \label{I:G(n)/I(n)} 
If $\mu_p \subset K$, then $G(n)/I(n)$ is a cyclic $p$-group.
\item \label{I:I_n is a p-group}
The group $I(n)$ is a $p$-group.
\item \label{I:I_n}
If $v(a)>v(c)$, then $I(n)$ is an elementary abelian $p$-group of order dividing $p^n$.
\item \label{I:G(infty)/I(infty)}
If $\mu_p \subset K$, then $G(\infty)/I(\infty) \cong \Z_p$.
\item \label{I:v(a-b)<0} If $v(a-b) < 0$, then $I(\infty)$ is an infinite pro-$p$ group;
if, moreover, $v(a)>v(c)$, then $I(\infty) \isom \left(\Z/p\Z\right)^{\infty}$.
\item \label{I:v(a-b)>=0} If $v(a-b) \ge 0$, then $I(\infty) = \{1\}$.
\end{enumerate}
\end{theorem}

\begin{proof}
Let $k_n$ be the residue field of $K_n$.
\begin{enumerate}[\upshape (a)]
\item The group $G(n)/I(n)$ is isomorphic to the group $\Gal(k_n/k)$,
a Galois group of an extension of finite fields, so it is cyclic.
By Lemma~\ref{L:G(n) is a p-group}, 
$G(n)$ is a $p$-group, so $G(n)/I(n)$ is a $p$-group too.
\item
Since $v(c)=-p/(p-1)$, the ramification index of $K$ over $\Q_p$ is divisible by $p-1$.
On the other hand, $\Q_p(\mu_p)/\Q_p$ is tamely ramified with ramification index $p-1$,
so Abhyankar's lemma implies that $K(\mu_p)/K$ is unramified.
Apply Lemma~\ref{L:G(n) is a p-group} with $K(\mu_p)$ in place of $K$.
\item
By Lemma~\ref{L:valuations of preimages}, if $m \ge 1$ and $\alpha \in f^{-m}(a)$,
then $v(\alpha)=-1/(p-1)$.

Next we prove by induction that for $n \ge 1$, 
for any distinct $\alpha_n,\beta_n \in f^{-n}(a)$,
we have $v(\alpha_n-\beta_n)=0$.
If $n=1$, then $\alpha_1 = \zeta \beta_1$ for some $p$th root of unity,
so $v(\alpha_1-\beta_1)=v(\zeta-1) + v(\beta_1) = 1/(p-1) - 1/(p-1) = 0$.
Now suppose that $n>1$ and the result holds for all $m<n$.
Given distinct $\alpha_n,\beta_n \in f^{-n}(a)$,
let $\alpha_{n-1}=f(\alpha_n)$ and $\beta_{n-1}=f(\beta_n)$.
Let $d=\alpha_{n-1}-\beta_{n-1}$ and $y=\beta_n$,
so $v(y) = -1/(p-1)$.
If $\alpha_{n-1} \ne \beta_{n-1}$,
then $v(d)=0$ by the inductive hypothesis, and $p v(y) + p/(p-1) = 0$ too,
so Lemma~\ref{L:differences between preimages}\eqref{I:small v(d)} 
implies that $v(\alpha_n-\beta_n) = v(d)/p = 0$.
If $\alpha_{n-1} = \beta_{n-1}$,
then $d=0$, so Lemma~\ref{L:differences between preimages}\eqref{I:large v(d)} applies: 
the solution to $f(x)-f(\beta_n)=0$ closest to $\beta_n$
is $\beta_n$ itself, and the other solutions satisfy $v(x-\beta_n) = v(y) + 1/(p-1) = 0$;
in particular, $v(\alpha_n-\beta_n)=0$.
In both cases, the inductive step is completed.

Let $n \ge 1$.
Let $\calO_n$ be the closed unit disk in $K_n$ centered at $0$;
let $\mm$ be the open unit disk in $K_n$ centered at $0$.
Let $D_n$ be the closed unit disk in $K_n$ containing $f^{-n}(a)$;
by the previous paragraph,
such a disk exists and the natural map $f^{-n}(a) \to D_n/\mm$ is injective.
Injectivity implies that $G(n)$ acts faithfully on $D_n/\mm$.
At this point, we use an argument parallel to that of 
the proof of Theorem~\ref{T:elementary abelian p-group}\eqref{I:G_n},
but using $I(n)$ instead of $G(n)$.
The translation action of $\calO_n/\mm$ on $D_n/\mm$ makes $D_n/\mm$ an $\calO_n/\mm$-torsor,
and this action is $G(n)$-equivariant and hence $I(n)$-equivariant.
Since $I(n)$ acts trivially on the residue field $\calO_n/\mm$,
we obtain a homomorphism $I(n) \to \Aut_{\textup{$\calO_n/\mm$-torsor}}(D_n/\mm) \isom \calO_n/\mm$.
Since $G(n)$ acts faithfully on $D_n/\mm$, this homomorphism is injective,
so $I(n)$ is an elementary abelian $p$-group.
The number of translations mapping $f^{-n}(a) \bmod \mm$ to itself
is at most $\# f^{-n}(a) = p^n$, so $\#I(n) \le p^n$.

\item
Fix $\alpha_n \in f^{-n}(a)$.
As $\beta_n$ varies over $f^{-n}(a)$,
the argument in the proof of~\eqref{I:I_n} shows that 
the differences $\alpha_n-\beta_n$ have valuation~$0$ and have distinct residues.
Thus $\#k_n \ge p^n$.
Hence $k_\infty$ is infinite, so $G(\infty)/I(\infty)$ is infinite.
On the other hand, by \eqref{I:G(n)/I(n)}, $G(\infty)/I(\infty)$ is an inverse limit of cyclic $p$-groups.
Thus $G(\infty)/I(\infty) \isom \Z_p$.

\item
By \eqref{I:I_n is a p-group} and~\eqref{I:I_n}, it will suffice to show that $I(\infty)$ is infinite.
Examining the Newton polygon of $x^p-x-c$ shows that $v(b)=v(c)/p = -1/(p-1)$.
We prove by induction that for each $n \ge 0$, 
each $\alpha_n \in f^{-n}(a)$ satisfies $v(\alpha_n-b) = v(a-b)/p^n < 0$.
The $n=0$ case is given.
Now suppose that $n \ge 1$, and the $n-1$ case for $\alpha_{n-1} = f(\alpha_n)$ is known.
Since $p v(b) - \nu_\infty = 0$, 
applying Lemma~\ref{L:differences between preimages}\eqref{I:small v(d)} with $(d,y)\colonequals (\alpha_{n-1}-b,b)$
and  $f(b)=b$
shows that the solution $\alpha_n$ to $f(x) = \alpha_{n-1}$ satisfies 
$v(\alpha_n-b) = v(\alpha_{n-1}-b)/p = v(a-b)/p^n < 0$,
which completes the inductive step.
Thus the ramification index of $K(f^{-n}(a))$ over $K$ tends to $\infty$ as $n \to \infty$.
\item
Let $\epsilon = a-b$, so $v(\epsilon) \ge 0$.
Then $v(a)=v(b+\epsilon)=v(b)=-1/(p-1)$.
Define conjugate polynomials $g(x) = f(z+b) - b$ and $h(y) = g(y+\epsilon) - \epsilon = f(z+a) - a \in K[y]$.
Then
\[
g(x) = x^p + \binom{p}{1} b x^{p-1} + \cdots + \binom{p}{p-1} b^{p-1} x.
\]
Since $v(b)=-1/(p-1)$, the polynomial $g(x)$ has $p$-adically integral coefficients,
and $g'(x)$ reduces modulo the maximal ideal to a nonzero constant.
Since $v(\epsilon) \ge 0$, the polynomial $h(y)$ has the same properties.
Thus adjoining solutions to $h(y)=e$ for any $p$-adically integral $e$ yields an unramified extension.
By induction, $K(h^{-n}(0))$ is unramified over $K$ for every $n \ge 0$.
Conjugating back shows that $K(f^{-n}(a))$ is unramified over $K$ for every $n \ge 0$.
Thus $I(\infty)=\{1\}$.
\qedhere
\end{enumerate}
\end{proof}

\begin{corollary}
\label{C:infinite extension}
If $\ell=p$ and $v(c)=-p/(p-1)$, then $[K_\infty:K] = \infty$.
\end{corollary}

\begin{proof}
We may replace $K$ by $K(\mu_p)$.
Then Theorem~\ref{T:-p/(p-1)}\eqref{I:G(infty)/I(infty)} implies that $G(\infty)/I(\infty)$ is infinite,
so $[K_\infty:K]=\#G(\infty) = \infty$.
\end{proof}

\begin{example}
\label{E:-1/4}
Let $p=2$ and $c=-1/4$, so $f(z)$ is $z^2+1/4$.
If $a=1/2$, then $K_\infty$ is the unramified $\Z_2$-extension of $\Q_2$.
\end{example}

\subsection{Ramification group lemmas}

We will prove results about the ramification groups of $G(\infty)$,
but first we need some lemmas about ramification groups in general.
Let $K$ be a local field, and let $L$ be a finite Galois extension of $K$
with Galois group $G$.
Let $v_L \colon L \surjects \Z \union \{\infty\}$ be the discrete valuation
normalized to have value group $\Z$.
Let $\calO_L \colonequals \{x \in L : v_L(x) \ge 0\}$.
In numbering ramification groups, we follow the conventions of \cite{SerreLocalFields1979}*{IV},
which we now recall.
For $u \in \R_{\ge 0}$, define the $u$th ramification group in the lower numbering by
\[
G_u  \colonequals \{ \sigma \in G : v_L({}^\sigma x - x) \ge u+1 \textup{ for all $x \in \calO_L$}\}.
\]
Define the Herbrand bijection $\R_{\ge 0} \to \R_{\ge 0}$ by 
\[
\phi_{L/K}(u) \colonequals \int_0^u \frac{dt}{(G_0:G_t)}.
\]
For $w \in \R_{\ge 0}$, define the $w$th ramification group $G^w$ in the upper numbering 
so that $G^{\phi_{L/K}(u)} = G_u$.
The lower and upper numbering ramification groups define descending filtrations of $G$.
The upper numbering is compatible with quotients, so for an infinite Galois extension $L$ of $K$ with Galois group,
we may define $G^w \colonequals \varprojlim \Gal(L'/K)^w$ as $L'$ ranges over the finite Galois extensions of $K$
contained in $L$.

\begin{lemma}
\label{L:intersection of ramification groups}
Let $K$ be a local field, and let $L$ be a Galois extension of $K$
with Galois group $G$.
Then $\Intersection_{w \in \R_{\ge 0}} G^w = \{1\}$.
\end{lemma}

\begin{proof}
The intersection maps to the corresponding intersection for each
finite Galois subextension $L'$ over $K$, so we may assume
that $L$ is finite over $K$.
Suppose that $\sigma \in \Intersection_{w \in \R_{\ge 0}} G^w$.
The $G^w$ are the same as the $G_u$, only renumbered,
so $\sigma \in G_u$ for all $u \in \R_{\ge 0}$.
Then for any $x \in \calO_L$,
we have $v_L({}^\sigma x - x) \ge u+1$ for all $u$,
so ${}^\sigma x = x$.
The field generated by the elements of $\calO_L$ is $L$,
so $\sigma = 1$ in $\Gal(L/K)$.
\end{proof}

\begin{lemma}
\label{L:comparing ramification groups}
Consider a tower of extensions $K \subseteq L \subseteq M$
of a local field $K$.
Suppose that $M$ is Galois over $K$ and $L$ is finite over $K$.
Let $G = \Gal(M/K)$ and $H = \Gal(M/L)$.
Then $G^w \intersect H \le H^w$ for all $w \in \R_{\ge 0}$.
\end{lemma}

\begin{proof}
If the result holds for every finite Galois extension of $K$
lying between $L$ and $M$,
then the result holds for $M$ too.
Thus we may assume that $M$ is finite over $K$.
Lower numbering ramification groups are compatible with subgroups;
that is, $H_t = G_t \intersect H$ for all $t \in \R_{\ge 0}$.
Thus $H_0/H_t$ injects into $G_0/G_t$, so 
\[
	\phi_{M/K}(u) 
	\colonequals \int_0^u \frac{dt}{(G_0:G_t)} 
	\le \int_0^u \frac{dt}{(H_0:H_t)} 
	\equalscolon \phi_{M/L}(u).
\]
Since the groups $G^w$ decrease as $w$ increases, for $s \in H$, this implies 
\[
	s \in G^{\phi_{M/L}(u)} 
	\implies s \in G^{\phi_{M/K}(u)} 
	\iff s \in G_u 
	\iff s \in H_u
	\iff s \in H^{\phi_{M/L}(u)}.
\]
Hence $G^{\phi_{M/L}(u)} \intersect H \le H^{\phi_{M/L}(u)}$.
As $u$ ranges over $[0,\infty)$, so does $\phi_{M/L}(u)$;
thus $G^w \intersect H \le H^w$ for all $w \in \R_{\ge 0}$.
\end{proof}

\begin{corollary}
\label{C:trivial ramification groups in towers}
With notation as in Lemma~\ref{L:comparing ramification groups},
suppose in addition that $L$ is Galois over $K$.
Let $w \in \R_{\ge 0}$.
If $H^w$ and $(G/H)^w$ are $\{1\}$, then $G^w=\{1\}$.
\end{corollary}

\begin{proof}
The surjection $G \surjects G/H$ maps $G^w$ into $(G/H)^w=\{1\}$,
so $G^w \le H$.
In particular, $G^w = G^w \intersect H$,
which by Lemma~\ref{L:comparing ramification groups}
is contained in $H^w = \{1\}$.
\end{proof}

\begin{corollary}
\label{C:triviality of ramification groups}
With notation as in Lemma~\ref{L:comparing ramification groups},
if $H^w=\{1\}$ for some $w \in \R_{\ge 0}$,
then $G^{w'}=\{1\}$ for some $w' \in \R_{\ge 0}$.
\end{corollary}

\begin{proof}
By Lemma~\ref{L:comparing ramification groups},
$G^w \intersect H \le H^w = \{1\}$.
Thus $G^w$ injects into the finite set $G/H$,
so $G^w$ is finite.
The groups $G^{w'}$ are decreasing and their intersection is $\{1\}$
by Lemma~\ref{L:intersection of ramification groups},
so $G^{w'}=\{1\}$ for some $w' \ge w$.
\end{proof}

\begin{lemma}
\label{L:ramification group in compositum}
Let $K$ be a local field.
Let $L_1,\ldots,L_n$ be Galois extensions of $K$.
Let $w \in \R_{\ge 0}$.
If $\Gal(L_i/K)^w = \{1\}$ for all $i$,
then $\Gal(L_1 \cdots L_n/K)^w = \{1\}$.
\end{lemma}

\begin{proof}
The injection $\Gal(L_1 \cdots L_n/K) \injects \prod_{i=1}^n \Gal(L_i/K)$
maps $\Gal(L_1 \cdots L_n/K)^w$ into each $\Gal(L_i/K)^w$.
\end{proof}

\begin{lemma}
\label{L:upper vs lower}
Let $L \supseteq K$ be a finite Galois extension of local fields
with Galois group $G$.
Then for any $u \in \R_{\ge 0}$, 
the $u$th upper and lower numbering ramification groups satisfy 
$G^u \le G_u$.
\end{lemma}

\begin{proof}
We have 
$\phi_{L/K}(u) \colonequals \int_0^u \frac{dt}{(G_0:G_u)} \le \int_0^u dt = u$,
so $G^u \le G^{\phi_{L/K}(u)} = G_u$.
\end{proof}

\subsection{Ramification groups of iterates}

We now return to the study of the Galois groups of $f^n(z)-a$.
The following theorem shows that when $v(c)=-p/(p-1)$,
the ramification in $K_\infty/K$ is not very deep.
Let $b \in \Kbar$ be a fixed point of $f$.
Let $e$ be the ramification index of $K$ over $\Q_p$.

\begin{theorem}
\label{T:trivial ramification group for -p/(p-1)}
Suppose that $\ell=p$.
If $v(c)=-p/(p-1)$, then there exists $w \in \R_{\ge 0}$ 
such that $G(\infty)^w=\{1\}$.
\end{theorem}

\begin{proof}
First suppose that $v(a) > v(c)$ and $b \in K$.
If $v(a-b) \ge 0$, then 
Theorem~\ref{T:-p/(p-1)}\eqref{I:v(a-b)>=0} implies that $I(\infty) = \{1\}$,
so the conclusion holds trivially, with $w=0$.
So assume that $v(a-b)<0$.
Let $n \ge 1$.
We have $v_{K_n} = (e \#I(n)) v$.
By Theorem~\ref{T:-p/(p-1)}\eqref{I:I_n}, we have $\#I(n) \le p^n$.
Let $K'$ be the maximal unramified extension of $K$ in $K_n$.
Fix $\alpha \in f^{-n}(a)$, and let $\gamma = \alpha-b$.
Let $\sigma \in I(n) = \Gal(K_n/K')$ be such that $\sigma \ne 1$.
The proof of Theorem~\ref{T:-p/(p-1)}\eqref{I:I_n} 
shows that $\sigma$ acts on $f^{-n}(a)$ without fixed points.
In particular, ${}^\sigma \alpha \ne \alpha$,
and the proof of Theorem~\ref{T:-p/(p-1)}\eqref{I:I_n}
shows that $v({}^\sigma \alpha - \alpha)=0$.
Since $\sigma$ fixes $b$, we obtain $v({}^\sigma \gamma - \gamma)=0$.
The proof of Theorem~\ref{T:-p/(p-1)}\eqref{I:v(a-b)<0}
shows that $v(\gamma) = v(a-b)/p^n$, which is negative, 
so $v_{K_n}(\gamma) = -(e \#I(n)) |v(a-b)|/p^n \ge -e |v(a-b)|$.
Since ${}^\sigma \gamma^{-1} - \gamma^{-1} = -({}^\sigma \gamma - \gamma)/({}^\sigma \gamma \cdot \gamma)$,
we have 
\[
	v_{K_n}({}^\sigma \gamma^{-1} - \gamma^{-1})
	\le 2 e |v(a-b)|.
\]
Hence for any positive integer $w \ge 2 e |v(a-b)|$, 
we have $G(n)_w = \{1\}$, so Lemma~\ref{L:upper vs lower} 
shows that $G(n)^w = \{1\}$ too.
This holds for all $n$, so $G(\infty)^w = \{1\}$ for such $w$.

Now we consider the general case.
By Lemma~\ref{L:valuations of preimages},
we can find $m \ge 1$ such that all $\alpha \in f^{-m}(a)$ satisfy
$v(\alpha)=v(c)/p$, so $v(\alpha) > v(c)$.
Let $L$ be a finite Galois extension of $K$ containing $f^{-m}(a)$ and $b$.
For each $\alpha$, the previous paragraph yields $w \in \R_{\ge 0}$
such that $\Gal(L(f^{-\infty}(\alpha))/L)^w = \{1\}$;
by taking the maximum of the $w$'s, we find one $w$ for which 
$\Gal(L(f^{-\infty}(\alpha))/L)^w = \{1\}$ for all $\alpha \in f^{-m}(a)$.
Taking the compositum over $\alpha$ yields
$\Gal(L(f^{-\infty}(a))/L)^w = \{1\}$
by Lemma~\ref{L:ramification group in compositum}.
By Corollary~\ref{C:triviality of ramification groups},
$\Gal(L(f^{-\infty}(a))/K)^{w'} = \{1\}$ for some $w' \in \R_{\ge 0}$.
Taking the image in the quotient $G(\infty)$ of $\Gal(L(f^{-\infty}(a))/K)$
shows that $G(\infty)^{w'}=\{1\}$.
\end{proof}

\begin{example}\label{E:shallow ramification}
Suppose that $\ell=p$ and $e=p-1$ and $v(c)=-p/(p-1)$ 
and $b\in K$ and $v(a-b)=-1/(p-1)$ (this implies $v(a) \ge -1/(p-1) > v(c)$).
Then the first paragraph of the proof of 
Theorem~\ref{T:trivial ramification group for -p/(p-1)} 
shows that $G(n)_2=\{1\}$ for all $n$.
On the other hand, $G(n)_0=G(n)_1$ since the inertia group
is of $p$-power order.
Thus the only break in the ramification filtration 
(in either the lower or upper numbering) occurs at $1$,
and for the upper numbering this holds also for $I(\infty)$.
\end{example}

\begin{remark}
\label{R:Sen}
Let $K$ be a characteristic~$0$ local field with perfect residue field
of characteristic~$p$.
For a continuous homomorphism $\rho$ from $\Gal(K_s/K)$ to a $p$-adic Lie group $G$,
Sen's theorem \cite{Sen1972}*{\S4} relates the ramification filtration 
to the ``Lie filtration'' of $G$.
Theorem~\ref{T:trivial ramification group for -p/(p-1)} 
and Example~\ref{E:shallow ramification}
show that the analogue for arboreal representations does not hold.
\end{remark}

\section{Insufficiently negative valuation}
\label{S:insufficiently negative valuation}

\begin{theorem}
\label{T:slightly negative v(c)}
If $\ell=p$ and $-p/(p-1) < v(c) < 0$, then $K_\infty/K$ is infinitely wildly ramified.
\end{theorem}

\begin{proof}
By Lemma~\ref{L:valuations of preimages}, we may replace $a$ by some iterated preimage
to assume that $v(\alpha) = v(c)/p$ for every $\alpha \in f^{-n}(a)$ for every $n \ge 0$.
Let $\alpha_0 = a$, and inductively choose $\alpha_n \in f^{-1}(\alpha_{n-1})$ for $n \ge 1$.
Let $\beta_0 = a$,  and inductively choose $\beta_n \in f^{-1}(\beta_{n-1})$ 
such that $\beta_1 \ne \alpha_1$.
Let $d_n = \beta_n - \alpha_n$.
By Lemma~\ref{L:differences between preimages}\eqref{I:large v(d)} with $d=0$ and $y=\alpha_1$,
we have $v(d_1) = v(c)/p + 1/(p-1) > 0$.

We prove by induction that $v(d_n) = v(d_1)/p^{n-1}$ for all $n \ge 1$.
The base case $n=1$ is trivial.
Suppose that $n \ge 2$ and the result holds for $n-1$.
Let $d = d_{n-1}$ and $y=\alpha_n$.
By the inductive hypothesis, 
\[
   v(d) = v(d_1)/p^{n-2} \le v(d_1) = v(c)/p + 1/(p-1) < p(v(c)/p+1/(p-1)) = p v(y) + p/(p-1).
\]
By Lemma~\ref{L:differences between preimages}\eqref{I:small v(d)}, 
$v(d_n) = v(d)/p = v(d_{n-1})/p = v(d_1)/p^{n-1}$.

Thus the exponent of $p$ in the denominator of $v(d_n)$ eventually grows with $n$,
so $K_\infty/K$ is infinitely wildly ramified.
\end{proof}

We next bound the growth rate of $[K_n:K]$.
We have $\mu_p \subseteq K_1$.
For $r \ge 1$, the field $K_{r+1}$ is obtained from $K_r$ by adjoining the $p$th roots
of the $p^r$ numbers $\alpha_r+c$ as $\alpha_r$ ranges over the elements of $f^{-r}(a)$.
By Kummer theory, $[K_{r+1}:K_r]$ equals the order of the subgroup
generated by these $p^r$ numbers in $K_r^\times/K_r^{\times p}$.
In particular, $[K_{r+1}:K_r] \le p^{p^r}$ for all $r \ge 1$.
Similarly, $[K_1:K(\mu_p)] \le p$.
Also, $[K(\mu_p):K] \le p-1$.
Taking the product yields the ``trivial'' bound
\[
[K_n:K] \le B_n \colonequals (p-1) \prod_{m=0}^{n-1} p^{p^m}.
\]
(If $p=2$, then $B_n = \# \Aut T_n$.  For any $p$, a $p$-Sylow subgroup of $\Aut T_n$ has order $\prod_{m=0}^{n-1} p^{p^m}$.)
The next theorem shows that when $v(c) < 0$, we can do better. 

\begin{theorem}
\label{T:group size}
Suppose that $\ell=p$ and $v(c)<0$.
Let $r \in \Z_{\ge 1}$ be such that $v(c) < -p/((p^r-1)(p-1))$.
Then there exists a constant $C$ depending on $p$, $r$, and $v(a)$ such that 
\[
[K_n:K] \le C B_n^{1-p^{-r}}.
\]
\end{theorem}

We will need the following lemma in the proof of Theorem~\ref{T:group size}.
\begin{lemma}
\label{L:close to 1}
Let $\epsilon \in K$.
If $v(\epsilon)>p/(p-1)$, then $1+\epsilon \in K^{\times p}$.
\end{lemma}

\begin{proof}
The hypothesis implies that the Newton polygon of $(1+x)^p - (1+\epsilon)$
has vertices at $(0,v(\epsilon))$, $(1,1)$, and $(p,0)$.
The width~$1$ segment at the left corresponds to a root in $K$.
\end{proof}

\begin{proof}[Proof of Theorem~\ref{T:group size}]
By Lemma~\ref{L:valuations of preimages},
there exists $m_0 \ge 1$ such that if $m \ge m_0$ and $\alpha_m \in f^{-m}(a)$,
then $v(\alpha_m) \ge v(c)$.

We will show that if $m \ge m_0$ and $\alpha_m \in f^{-m}(a)$, then
\begin{equation}
\label{E:product is pth power}
\prod_{\alpha_{m+r} \in f^{-r}(\alpha_m)} (\alpha_{m+r}+c) \in K_{m+r}^{\times p}.
\end{equation}
The numbers $\alpha_{m+r}+c$ in the product are the zeros of the polynomial $f^r(x-c)-\alpha_m$.
Their product is $(-1)^{p^r}$ times the constant term, so the product is
\begin{equation}
\label{E:product of numbers}
(-1)^{p^r} (f^r(-c) - \alpha_m) = (-1)^{p^r}(t^p - c - \alpha_m) = ((-1)^{p^{r-1}} t)^p \left( 1 - \frac{c+\alpha_m}{t^p} \right),
\end{equation}
where $t \colonequals f^{r-1}(-c)$.
We have $v(t) = p^{r-1} v(c)$, and $v(c+\alpha_m) \ge v(c)$, so $v((c+\alpha_m)/t^p) \ge v(c) - p^r v(c) > p/(p-1)$.
Thus, by Lemma~\ref{L:close to 1} over $K_{m+r}$, the second factor on the right of \eqref{E:product of numbers} 
is a $p$th power in $K_{m+r}$ (as is the first).
This proves~\eqref{E:product is pth power}.

Applying~\eqref{E:product is pth power} to the $p^m$ numbers $\alpha_m \in f^{-m}(a)$
shows that $K_{m+r+1}$ is obtained from $K_{m+r}$ by adjoining at most $p^m(p^r-1)$ roots,
so 
\[
[K_{m+r+1}:K_{m+r}] \le p^{p^m(p^r-1)} = p^{p^{m+r}(1-p^{-r}) }.
\]
Thus if $n \ge m_0+r$,
\[
[K_n:K] \le [K_{m_0+r}:K] \prod_{s=m_0+r}^{n-1} p^{p^s(1-p^{-r}) } \le C B_n^{1-p^{-r}}
\]
for some $C$.
\end{proof}


\section{Nonnegative valuation}
\label{S:nonnegative valuation}

In this section, we treat the tame and wild cases in which $v(c) \ge 0$.
Fix an arbitrary sequence of preimages $(\alpha_n)_{n \ge 0}$
defined by 
$\alpha_0 \colonequals a$ 
and $\alpha_{n+1} \in f^{-1}(\alpha_n)$ for $n \ge 0$.
Let $(\beta_n)_{n \ge 0}$ be another such sequence;
if $a+c \ne 0$, we may assume that $\beta_1 \ne \alpha_1$.
For $n \ge 0$, let $d_n = \alpha_n-\beta_n$.

\begin{lemma}
\label{L:stupid lemma}
If $v(c) \ge 0$ and $\min\{v(a),v(c)\} \ne 0$ and $v(a) \ne v(c)$, 
then $K_\infty/K$ is infinitely ramified, and infinitely wildly ramified if $\ell=p$.
\end{lemma}

\begin{proof}
We prove $v(\alpha_n) = \min\{v(a),v(c)\}/\ell^n < v(c)$ for $n \ge 1$
by induction.
The equation $\alpha_1^\ell - c = a$ implies that 
$v(\alpha_1) = \min\{v(a),v(c)\}/\ell < v(c)$.
If the statement is true for a given $n \ge 1$,
then the equation $\alpha_{n+1}^\ell - c = \alpha_n$
implies $v(\alpha_{n+1}) = v(\alpha_n)/\ell$,
so $v(\alpha_{n+1}) = \min\{v(a),v(c)\}/\ell^{n+1} < v(c)$.
Thus the denominator of $v(\alpha_n)$ tends to infinity,
so $K_\infty/K$ is infinitely ramified.
If $\ell=p$, the proof shows also that the exponent of $p$ in the denominator of $v(\alpha_n)$ tends to infinity.
\end{proof}

\subsection{Wild case}

We now assume that $\ell=p$ (and $v(c) \ge 0$).
The following will be used to prove the main result of this section, Theorem~\ref{T:nonnegative v(c)}.

\begin{lemma}
\label{L:stupid lemma 2}
If $\ell=p$ and $v(c) > 0$ and $v(a)=0$,
then $K_\infty/K$ is infinitely wildly ramified.
\end{lemma}

\begin{proof}
By induction, $v(\beta_n)=0$ for all $n \ge 0$.
Now $v(d_1) = 1/(p-1)$ 
by Lemma~\ref{L:differences between preimages}\eqref{I:large v(d)}
with $d=d_0$ and $y=\beta_1$.
Then $v(d_n)=v(d_1)/p^{n-1}$ by induction on $n$, 
by Lemma~\ref{L:differences between preimages}\eqref{I:small v(d)}
with $d=d_{n-1}$ and $y=\beta_n$.
Thus the denominator of $v(d_n)$ tends to infinity,
so $K_\infty/K$ is infinitely ramified.
\end{proof}

\begin{theorem}
\label{T:nonnegative v(c)}
If $\ell=p$ and $v(c) \ge 0$, then $K_\infty/K$ is infinitely wildly ramified.
\end{theorem}

\begin{proof}
Lemmas \ref{L:stupid lemma} and \ref{L:stupid lemma 2}
apply unless $v(a) > v(c) = 0$ or $v(a) = v(c) \ge 0$.
If $v(a)> v(c) = 0$, then $v(\alpha_1)=0$.
So by replacing $a$ by $\alpha_1$ if necessary,
we may assume that $v(a)=0$.
Thus it remains to consider the case $v(a)=v(c) \ge 0$.
If any iterated preimage of $a$ has valuation not $v(c)$, 
then we reduce to a previous case.

So assume that $v(\alpha_n) = v(c)$ for all $n \ge 1$.
We now prove $v(d_n) = (v(c) + 1/(p-1))/p^{n-1}$ for $n \ge 1$
by induction.
First, $v(d_0) = \infty > \ell v(\alpha_1) - \nu_\infty$,
so Lemma~\ref{L:differences between preimages}\eqref{I:large v(d)}
implies $v(d_1) = v(\alpha_1) + 1/(p-1) = v(c) + 1/(p-1) > 0$.
Next, for $n \ge 2$, by the inductive hypothesis, 
$v(d_{n-1}) \le v(d_1) \le p v(c) + p/(p-1) = p v(\alpha_n) - \nu_\infty$,
so Lemma~\ref{L:differences between preimages}\eqref{I:small v(d)}
implies $v(d_n) = v(d_{n-1})/p = (v(c) + 1/(p-1))/p^{n-1}$.
Thus the denominator of $v(d_n)$ tends to infinity,
so $K_\infty/K$ is infinitely ramified.
\end{proof}

\subsection{Tame case}

We now assume that $p \nmid \ell$ (and $v(c) \ge 0$).
Lemma~\ref{L:stupid lemma} handles the case where $v(a) < 0$,
and Theorem~\ref{T:tame v(c)>=0} below will handle the case where $v(a) \ge 0$.
Let $\mm$ (resp.\ $\mm_s$) 
be the maximal ideal of the valuation ring $\calO$ in $K$ (resp.\ $K_s$).
We say that an element $u \in K$ is \defi{periodic} (for $f$) if $f^n(u)=u$ for some $n \ge 1$,
\defi{preperiodic} if $f^m(u)$ is periodic for some $m \ge 0$, 
and \defi{strictly preperiodic} if it is preperiodic but not periodic.
If $u$ is periodic, its \defi{period} is the smallest $n \ge 1$ such that $f^n(u)=u$.
We say that $u,w \in K$ are \defi{in a single cycle} if $u$ is periodic and there exists $n \ge 0$
such that $f^n(u)=w$.
These notions apply also to dynamics of a polynomial map defined over the residue field $\calO/\mm$.

\begin{theorem}
\label{T:tame v(c)>=0}
Suppose that $v(c) \ge 0$ and $v(a) \ge 0$.
\begin{enumerate}[\upshape (a)]
\item \label{I:not in forward orbit mod m}
If $a \bmod \mm$ is not in the forward orbit of $0 \bmod \mm$,
then $K_\infty/K$ is unramified.
\item  \label{I:strictly preperiodic mod m}
If $0 \bmod \mm$ is strictly preperiodic mod $\mm$,
then the ramification index of $K_\infty/K$ divides $\ell$.
\item \label{I:single cycle}
If $0$ and $a$ are in a single cycle, then $K_\infty/K$ is unramified.
\item \label{I:single cycle mod m}
If $0 \bmod \mm$ and $a \bmod \mm$ are in a single cycle mod $\mm$,
but $0$ and $a$ are not both in a single cycle,
then $K_\infty/K$ is infinitely ramified.
\end{enumerate}
\end{theorem}

Parts \eqref{I:not in forward orbit mod m} and~\eqref{I:strictly preperiodic mod m}
cover the cases where $0 \bmod \mm$ and $a \bmod \mm$ are \emph{not} in a single cycle mod $\mm$.
Parts \eqref{I:single cycle} and~\eqref{I:single cycle mod m}
cover the cases where $0 \bmod \mm$ and $a \bmod \mm$ \emph{are} in a single cycle mod $\mm$.

\begin{proof}\hfill
\begin{enumerate}[\upshape (a)]
\item
In taking preimages, we are taking $\ell$th roots of units only,
so the extensions are unramified.
\item
For any sequence of preimages $(\alpha_n)_{n \ge 0}$ 
with $\alpha_0=a$ and $f(\alpha_{n+1})=\alpha_n$ for all $n$,
the extension $K(\alpha_0,\alpha_1,\ldots)$ is tamely ramified
of ramification index dividing $\ell$,
since the sequence is obtained by adjoining $\ell$th roots
of elements such that at most one of them is a non-unit
(otherwise $0 \bmod \mm$ would have been periodic).
The field $K_\infty$ is the compositum of these extensions,
so it too is tamely ramified of ramification index dividing $\ell$.
\item
Let $C_0$ be the cycle containing $0$ and $a$.
Let $n$ be the length of $C_0$.
Let $\alpha \in C_0$.
Then $(f^n)'(\alpha)=\prod_{\beta \in C_0} f'(\beta) = 0$,
since $f'(0)=0$.
Thus the derivative of $f^n(x)-x$ at $\alpha$ is $-1$.
By Hensel's lemma, $f^n(x)-x$ has a unique solution in $K$
congruent to $\alpha$ modulo $\mm$.
This applies to every $\alpha \in C_0$,
so the elements of $C_0$ are distinct modulo~$\mm$.

Suppose that $\beta \in K_s$ is an iterated preimage of $a$.
Since $a \in C_0$, there exists $r \ge 0$ such that $f^r(\beta) \in C_0$.
We claim that if $\beta \equiv \alpha \pmod{\mm_s}$ for some $\alpha \in C_0$,
then $\beta = \alpha$.
We use induction on $r$.
If $r=0$, then $\beta \in C_0$, so the previous paragraph implies that 
$\beta=\alpha$.
If $r \ge 1$, then the inductive hypothesis applied to 
$f(\beta) \equiv f(\alpha) \pmod{\mm_s}$ shows that $f(\beta)=f(\alpha)$.
Then $\beta = \zeta \alpha$ for some $\ell$th root of unity $\zeta$.
Thus $\zeta \alpha \equiv \alpha \pmod{\mm_s}$.
If $\alpha \equiv 0 \pmod{\mm_s}$, then $\alpha=0$ by the previous paragraph,
so $\beta = \zeta \alpha = 0 = \alpha$.
Otherwise, $\zeta \equiv 1 \pmod{\mm_s}$.
Since $\ell \ne \Char k$, this implies $\zeta=1$, so $\beta=\alpha$.

The claim shows that all iterated preimages of $a$ that are $0 \bmod \mm_s$
are equal to $0$.
Thus in taking preimages, we are taking $\ell$th roots of units and $0$ only,
so the extensions are unramified.
\item
Let $m$ be the period of $0 \bmod \mm$.
The derivative of $f^m(x)-x \bmod \mm$ at $0$ is $-1$, a unit, 
so by Hensel's lemma, 
there is a unique solution to $f^m(x)-x=0$ that reduces to $0 \bmod \mm$;
call it $b$.

Since $0 \bmod \mm$ and $a \bmod \mm$ are in a single cycle mod $\mm$,
we may choose a sequence of preimages $(\alpha_n)$
(with $\alpha_0=a$ and $f(\alpha_{n+1})=\alpha_n$ for all $n$)
such that $\alpha_n \equiv 0 \bmod \mm_s$ for infinitely many $n$.
We may assume that no $\alpha_n$ is equal to $b$: 
choose the $\alpha_i$ one at a time, and if one of them is $b$,
multiply it by a nontrivial $\ell$th root of unity before proceeding;
this changes it because if $\alpha_i=b=0$, then $0$ is periodic (since $b$ is)
and $a$ is in the forward orbit of $0$ (since $a$ is in the forward orbit of $\alpha_i$, but then $0$ and $a$ would belong to a single cycle, 
contradicting our hypothesis).
Let $\beta_0$, $\beta_1$, \dots be all the numbers in the 
sequence $(\alpha_n)$ that are $0 \bmod \mm_s$.
Thus $f^m(\beta_{i+1})=\beta_i$ for all $i$.

We now prove that $0 < v(\beta_{i+1}-b) < v(\beta_i-b)$ for all $i$.
Let $\epsilon = \beta_{i+1}-b$, so $v(\epsilon)>0$.
We have $f^m(b+x) = b + (f^m)'(b) x + x^2 R(x)$ for some $R(x) \in \calO[x]$.
Substituting $x=\epsilon$ yields
$\beta_i = b + (f^m)'(b) \epsilon \pmod{\epsilon^2}$.
Since $v(\epsilon)>0$ and $v((f^m)'(b))>0$, 
we obtain $v(\beta_i-b) > v(\epsilon) = v(\beta_{i+1}-b) > 0$.

This holds for all $i$, so $K_\infty/K$ is infinitely ramified.\qedhere
\end{enumerate}
\end{proof}

\section{Real case}
\label{S:real case}

\begin{theorem}
\label{T:real}
Let $f(z) = z^k - c \in \R[z]$ for some $k \ge 2$ and $c \in \R^\times$.
Given $a \in \R$, define $K_\infty$ as before.
\begin{enumerate}[\upshape (a)]
\item If $k>2$, then $K_\infty = \C$.
\item \label{I:k=2 and c<2}
If $k=2$ and $c < 2$, then $K_\infty = \C$.
\item If $k=2$ and $c \ge 2$, then $K_\infty$ is $\R$ or $\C$ according to whether
$a \in [-c,c^2-c]$ or not, respectively.
\end{enumerate}
\end{theorem}

\begin{proof}
\hfill
\begin{enumerate}[\upshape (a)]
\item There exists a nonzero $\beta \in f^{-n}(a)$ for some $n \ge 1$,
since otherwise $c=0$.
Then for every $k$th root of unity $\zeta$,
we have $\zeta \beta \in f^{-n}(a)$ too, so $\zeta = (\zeta \beta)/\beta \in K_\infty$.
Thus $K_\infty = \C$.
\item
Let $h(x) \colonequals \sqrt{c+x}$; if $x \ge -c$, take the nonnegative square root.
Thus $h(x)$ is strictly increasing on $[-c,\infty)$.

Suppose that $K_\infty=\R$.
Then all iterated preimages are real,
and in particular, $h^n(a) \in \R_{\ge 0}$ for all $n \ge 0$.
Also $c-h^n(a) \ge 0$ for all $n \ge 1$, since $-h^n(a)$ is a preimage of $h^{n-1}(a)$,
and $h(-h^n(a))=\sqrt{c-h^n(a)}$.
In particular, $c \ge c-h(a) \ge 0$.
We assumed $c \ne 0$, so $c>0$.

The fixed points of $f(z)$ are 
$L \colonequals (1+\sqrt{1+4c})/2 > 0$
and $L' \colonequals (1-\sqrt{1+4c})/2 < 0$.
The only solution to $h(x)=x$ in $[0,\infty)$ is $L$,
and $h$ is strictly increasing,
and $h(0)>0$ and $h(x)<x$ for large positive $x$;
thus $x \le h(x) \le L$ for $x \in [0,L]$, 
and $L \le h(x) \le x$ for $x \in [L,\infty)$.
In particular, $(h^n(a))_{n \ge 1}$ is a bounded monotonic sequence,
so it converges.
The limit is a nonnegative fixed point of $h$, so the limit is $L$.

On the other hand, the hypothesis $c<2$ implies that $L>c$, so
$h^n(a)>c$ for sufficiently large $n$.
This contradicts $c-h^n(a) \ge 0$.
\item
The hypothesis $c \ge 2$ implies that $L \le c$.
If $x \in [-c,c^2-c]$, then $c+x \ge 0$, and $\sqrt{c+x} \le \sqrt{c^2} = c$; also, $c \le c^2-c$,
so $h(x),-h(x) \in [-c,c^2-c]$.
Iterating shows that if $a \in [-c,c^2-c]$, 
then all iterated preimages are real, so $K_\infty=\R$.

If $a<-c$, then $h(a) \notin \R$, so $K_\infty=\C$.

If $a>c^2-c$, then $h(a) > c$, contradicting the inequality $c-h(a) \ge 0$ derived in the proof of~\eqref{I:k=2 and c<2}, so $K_\infty=\C$.
\qedhere
\end{enumerate}
\end{proof}

\section*{Acknowledgements} 

We thank Robert L. Benedetto, Robert Harron, and Yevgeny Zaytman.
We thank also the referees for many insightful suggestions. 
This research began at the workshop ``The Galois theory of orbits in arithmetic dynamics''
organized by Rafe Jones, Michelle Manes, and Joseph Silverman 
at the American Institute of Mathematics.

\end{document}